\documentclass[10pt,reqno]{amsart}
\usepackage{amssymb}
\usepackage{amsmath}
\usepackage{amsfonts}
\usepackage{amscd}
\usepackage{mathrsfs}
\usepackage{color}
\usepackage{bigints}

\newtheorem{thm}{Theorem}[section]

\newtheorem{lem}[thm]{Lemma}
\newtheorem{prop}[thm]{Proposition}

\theoremstyle{definition}

\theoremstyle{remark}

\theoremstyle{remark}

\numberwithin{equation}{section}

\newcommand{\n}[1]{\left\vert#1\right\vert}
\newcommand{\nm}[1]{\left\Vert#1\right\Vert}

\begin{document}

\title[]{\small\sf A non-integrated hypersurface defect relation for meromorphic maps over complete K\"{a}hler manifolds into projective algebraic varieties}

\author[]{\tt Wei Chen$^{\dagger}$ and Qi Han$^{\ddagger}$}

\address{$^{\dagger}$ {\sf Department of Mathematics, Shandong University, Jinan, Shandong 250100, P.R. China}
\vskip 2pt $^{\ddagger}$ {\sf Department of Mathematics, Texas A\&M University at San Antonio, San Antonio, Texas 78224, USA}
\vskip 2pt \hspace{1.5mm} {\rm Email: {\tt weichensdu@126.com} (W. Chen) and {\tt qhan@tamusa.edu} (Q. Han)}}
\thanks{{\sf 2010 Mathematics Subject Classification.} 32H30, 32H04, 32H25, 32A22.}
\thanks{{\sf Keywords.} K\"{a}hler manifold, projective algebraic variety, meromorphic map, Nevanlinna theory, hypersurface, $k$-subgeneral position, non-integrated defect relation.}

\begin{abstract}
  In this paper, a non-integrated defect relation for meromorphic maps from complete K\"{a}hler manifolds $M$ into smooth projective algebraic varieties $V$ intersecting hypersurfaces located in $k$-subgeneral position (see \eqref{e1.5} below) is proved.
  The novelty of this result lies in that both the upper bound and the truncation level of our defect relation depend only on $k$, $\dim_{\,\mathbb{C}}(V)$ and the degrees of the hypersurfaces considered; besides, this defect relation recovers Hirotaka Fujimoto \cite[Theorem 1.1]{Fu3} when subjected to the same condition.
\end{abstract}

\maketitle

\section{Introduction}\label{Int} 
\noindent Fujimoto \cite{Fu1,Fu2,Fu3,Fu5} introduced the innovative notion of non-integrated, or modified, defect for meromorphic maps over a complex K\"{a}hler manifold into the complex projective space.
Recent extensions and generalizations may be found in Ru and Sogome \cite{RS1}, as well as Tan and Truong \cite{TT}.
Below we will replicate the essential elements in this aspect from those references.

Denote $M$ an $m$-dimensional K\"{a}hler manifold with K\"{a}hler form $\boldsymbol{\omega}=\frac{\sqrt{-1}}{2}\sum_{i,j}h_{i\bar{j}}\,dz_i\wedge d{\bar{z}_j}$.
Write $\operatorname{Ric}\boldsymbol{\omega}=dd^c\log\left(\det(h_{i\bar{j}})\right)$ with $d=\partial+\bar{\partial}$ and $d^c=\frac{\sqrt{-1}}{4\pi}(\bar{\partial}-\partial)$.
Let $f:M\to\mathbb{P}^n(\mathbb{C})$ be a meromorphic map, and let $D$ be a hypersurface in $\mathbb{P}^n(\mathbb{C})$ of degree $d$ with $f(M)\nsubseteq D$.
Take $\nu^f_D$ to be the intersection divisor generated through $f$ and $D$, and take $\mu_0>0$ to be an integer.
Denote $\mathcal{A}\left(D,\mu_0\right)$ the family of constants $\eta\geq0$ such that there exists a bounded, nonnegative, continuous function $\mathtt{h}$ on $M$, with zeros of order no less than $\min\left\{\nu^f_D,\mu_0\right\}$, satisfying
\begin{equation}\label{e1.1}
d\,\eta\,\Omega_f+dd^c\log\mathtt{h}^2\geq\left[\min\left\{\nu^f_D,\mu_0\right\}\right].
\end{equation}
Here, $\Omega_f$ denotes the pull-back of the normalized Fubini-Study metric form on $\mathbb{P}^n(\mathbb{C})$, and $\left[\nu\right]$ denotes the $\left(1,1\right)$-current associated with the divisor $\nu\geq0$.

Note condition \eqref{e1.1} says that for each nonzero holomorphic function $\psi$ on an open set $U$ of $M$ with $\nu^0_\psi=\min\left\{\nu^f_D,\mu_0\right\}$ outside an analytic subset of codimension at least $2$, the function $\mathtt{v}:=\log\left(\frac{\mathtt{h}^2\nm{\mathfrak{f}}^{2d\eta}}{\n{\psi}^2}\right)$ is continuous and pluri-subharmonic, where $\nm{\mathfrak{f}}^2=\sum_{\imath=0}^n\n{\mathtt{f}_\imath}^2$ for a (local) reduced representation $\mathfrak{f}=\left(\mathtt{f}_0,\mathtt{f}_1,\ldots,\mathtt{f}_n\right):M\to\mathbb{C}^{n+1}$ of $f=\left[\mathtt{f}_0:\mathtt{f}_1:\cdots:\mathtt{f}_n\right]$.

The \textsl{non-integrated defect} of $f$ regarding $D$, truncated at level $\mu_0$, is defined as
\begin{equation}\label{e1.2}
\delta_{\mu_0}^f(D)=1-\inf\left\{\eta\geq0:\eta\in\mathcal{A}\left(D,\mu_0\right)\right\}.
\end{equation}
Then, like Nevanlinna's or Stoll's classical defects, $0\leq\delta_{\mu_0+1}^f(D)\leq\delta_{\mu_0}^f(D)\leq1$, $\delta_{\mu_0}^f(D)=1$ if
$f(M)\cap D=\emptyset$, and $\delta_{\mu_0}^f(D)\geq1-\frac{\mu_0}{\mu}$ for any integer $\mu\geq\mu_0$ if $\left[\nu^f_D-\mu\right]\geq0$ on $f^{-1}(D)$.

Further, we say $f:M\to\mathbb{P}^n(\mathbb{C})$ satisfies the ``{\bf condition $\mathbf{C}(\rho)$}'' provided for some constant $\rho\geq0$, there is a bounded, nonnegative, continuous function $\mathtt{h}$ on $M$ such that
\begin{equation}\label{e1.3}
\rho\,\Omega_f+dd^c\log\mathtt{h}^2\geq\operatorname{Ric}\boldsymbol{\omega}.
\end{equation}

Now, the original result of Fujimoto \cite[Theorem 1.1]{Fu3} can be stated as follows.

\begin{thm}\label{t1.1}
Assume $M$ is an $m$-dimensional complete K\"{a}hler manifold such that the universal covering of $M$ is biholomorphically isomorphic to a ball in $\mathbb{C}^m$.
Let $f:M\to\mathbb{P}^n(\mathbb{C})$ be a linearly non-degenerate meromorphic map such that the {\bf condition $\mathbf{C}(\rho)$} is satisfied, and let $H_1,H_2,\ldots,H_q$ be $q\left(\geq n+1\right)$ hyperplanes in $\mathbb{P}^n(\mathbb{C})$ that are located in general position.
Then, one has the following defect relation
\begin{equation}\label{e1.4}
\sum_{j=1}^q\delta_n^f(H_j)\leq n+1+\rho\,n\left(n+1\right).
\end{equation}
\end{thm}

Ru and Sogome \cite{RS1} (see also Yan \cite{Ya}), and Tan and Truong \cite{TT} generalized independently the preceding Theorem \ref{t1.1} in the way that $\mathbb{P}^n(\mathbb{C})$ is replaced by a projective algebraic variety $V\subseteq\mathbb{P}^N(\mathbb{C})$ and hyperplanes in $\mathbb{P}^n(\mathbb{C})$ located in general position are extended to hypersurfaces in $\mathbb{P}^N(\mathbb{C})$ located in different types of $k$-subgeneral positions.
One recalls that the $k$-subgeneral position condition used in \cite{TT} comes from Dethloff, Tan and Thai \cite[Definition 1.1]{DTT}.

It is noteworthy that both the upper bounds and the truncation levels of the defect relations obtained in \cite[Theorem 1.1]{RS1}, \cite[Definition 1.1 and Theorem 1.2]{TT} and \cite[Definition 1.2 and Theorem 1.1]{Ya} depend on a given constant $\epsilon>0$, and both blow up to $+\infty$ as $\epsilon\to0$.
Also, it's not clear to us if those results can recover Theorem \ref{t1.1} under the same assumptions.

In the sequel, assume that $V\subseteq\mathbb{P}^N(\mathbb{C})$ is a smooth projective algebraic variety of dimension $n\left(\leq N\right)$.
$q\left(>k\right)$ hypersurfaces $D_1,D_2,\ldots,D_q$ in $\mathbb{P}^N(\mathbb{C})$ are said to be located in \textsl{$k$-subgeneral position} $\left(k\geq n\right)$ with respect to $V$ provided for every $1\leq j_0<j_1<\cdots<j_k\leq q$,
\begin{equation}\label{e1.5}
\left(\bigcap_{s=0}^k\,\operatorname{supp}\,(D_{j_s})\right)\cap V=\emptyset
\footnote{As far as we can check, this condition \eqref{e1.5} appeared first in Chen, Ru and Yan \cite{CRY}.}.
\end{equation}
Here, $\operatorname{supp}\,(D)$ is the support of the divisor $D$.
One says $D_1,D_2,\ldots,D_q$ are in general position with respect to $V$, if they are located in $n$-subgeneral position with respect to $V$.

The purpose of this paper is by combining the techniques used in \cite{TT} and \cite{Ya} to describe a hypersurface defect relation, with definite truncation level and explicit upper bound, that will be exactly Fujimoto's original Theorem \ref{t1.1} when $d=1$, $k=n=N$ and $V=\mathbb{P}^n(\mathbb{C})$.

Fix an integer $d\geq1$.
Write $\mathcal{H}_d$ the vector space of homogeneous polynomials of degree $d$ in $\mathbb{C}\left[w_0,w_1,\ldots,w_N\right]$ and $\mathcal{I}_V$ the prime ideal in $\mathbb{C}\left[w_0,w_1,\ldots,w_N\right]$ defining $V$.
Denote
\begin{equation}
H_V(d):=\dim_{\,\mathbb{C}}\left(\frac{\mathcal{H}_d}{\mathcal{H}_d\cap\mathcal{I}_V}\right)\nonumber
\end{equation}
to be the Hilbert function of $V$.
$H_V(d)=n+1$ when $d=1$, $n=N$ and $V=\mathbb{P}^n(\mathbb{C})$.

Finally, we can formulate our main theorem of this paper as the following result.

\begin{thm}\label{t1.2}
Suppose $M$ is an $m$-dimensional complete K\"{a}hler manifold such that the universal covering of $M$ is biholomorphically isomorphic to a ball in $\mathbb{C}^m$, and assume $V\subseteq\mathbb{P}^N(\mathbb{C})$ is an irreducible projective algebraic variety of dimension $n\left(\leq N\right)$.
Let $f:M\to V$ be an algebraically non-degenerate meromorphic map such that the {\bf condition $\mathbf{C}(\rho)$} is satisfied, and let $D_1,D_2,\ldots,D_q$ be $q\left(\geq k+1\right)$ hypersurfaces in $\mathbb{P}^N(\mathbb{C})$ that are located in $k$-subgeneral position $\left(k\geq n\right)$ in regard to $V$ having degrees $d_1,d_2,\ldots,d_q$ respectively.
Denote by $d$ the least common multiple of $d_1,d_2,\ldots,d_q$.
Then, one has the following defect relation
\begin{equation}\label{e1.6}
\sum_{j=1}^q\delta_{H_V(d)-1}^f(D_j)\leq\frac{2k-n+1}{n+1}\left\{H_V(d)+\frac{\rho}{d}\,H_V(d)\left(H_V(d)-1\right)\right\}.
\end{equation}
\end{thm}

It is worthwhile to mention when $d=1$, $k=n=N$ and $V=\mathbb{P}^n(\mathbb{C})$, Theorem \ref{t1.2} recovers exactly Fujimoto's initial work.
As
$H_V(d)\leq\left(\begin{array}{ll}
d+N\\
\hspace{3mm}N\\
\end{array}\right)$,
the truncation level in Theorem \ref{t1.2} is smaller than that in \cite[Theorem 1.2]{TT} and also better than those in \cite{RS1,Ya}, yet the upper bound in \eqref{e1.6} might be larger than those in \cite{RS1,TT,Ya} (depending on their $\epsilon$).

\section{Preliminaries}\label{Pre} 
\noindent In this auxiliary section, we describe some basic notations and necessary results that are used afterwards throughout this paper.

Denote $\nm{z}^2=\sum_{\jmath=1}^m\n{z_\jmath}^2$ for $z=\left(z_1,z_2,\ldots,z_m\right)\in\mathbb{C}^m$.
Write $B(r)=\left\{z\in\mathbb{C}^m:\nm{z}<r\right\}$ and $S(r)=\left\{z\in\mathbb{C}^m:\nm{z}=r\right\}$ for $r\in\left(0,\infty\right)$, and $B(\infty)=\mathbb{C}^m$.
Define
\begin{equation*}
\begin{split}
&\upsilon_\jmath=\left(dd^c\nm{z}^2\right)^\jmath
\hspace{2mm}\mathrm{for}\hspace{2mm}\jmath=1,2,\ldots,m\hspace{2mm}\mathrm{on}\hspace{2mm}\mathbb{C}^m,\hspace{2mm}\mathrm{and}\\
&\sigma_m=d^c\log\nm{z}^2\wedge\left(dd^c\log\nm{z}^2\right)^{m-1}\hspace{2mm}\mathrm{on}\hspace{2mm}\mathbb{C}^m\setminus\left\{0\right\}.
\end{split}
\end{equation*}

Suppose $f:B(R_0)\to\mathbb{P}^n(\mathbb{C})$ is a meromorphic map with $0<R_0\leq\infty$.
Choose holomorphic functions $\mathtt{f}_0,\mathtt{f}_1,\ldots,\mathtt{f}_n$ with $\mathfrak{f}=\left(\mathtt{f}_0,\mathtt{f}_1,\ldots,\mathtt{f}_n\right):B(R_0)\setminus\mathtt{I}_f\to\mathbb{C}^{n+1}$ a reduced representation of $f$.
Notice the singularity set $\mathtt{I}_f:=\left\{z\in B(R_0):\mathtt{f}_0(z)=\mathtt{f}_1(z)=\cdots=\mathtt{f}_n(z)=0\right\}$ of $f$ is of dimension at most $m-2$.
Fix this reduced representation $\mathfrak{f}$ of $f$.
Then, $\Omega_f=dd^c\log\nm{\mathfrak{f}}^2$ will be the pull-back of the normalized Fubini-Study metric form on $\mathbb{P}^n(\mathbb{C})$ through $f$.

Given $r_0\in\left(0,R_0\right)$, the \textsl{characteristic function} of $f$ for $r\in\left(r_0,R_0\right)$ is defined as
\begin{equation}\label{e2.1}
T_f(r,r_0)=\int_{r_0}^r\frac{dt}{t^{2m-1}}\int_{B(t)}\Omega_f\wedge\upsilon_{m-1},
\end{equation}
which can also be written as
\begin{equation}\label{e2.2}
T_f(r,r_0)=\int_{S(r)}\log\nm{\mathfrak{f}}\sigma_m-\int_{S(r_0)}\log\nm{\mathfrak{f}}\sigma_m.
\end{equation}

For a holomorphic function $\psi$ on an open subset $U$ of $\mathbb{C}^m$ and $\alpha=\left(\alpha_1,\alpha_2,\ldots,\alpha_m\right)\in\mathbb{Z}^m_{\geq0}$, an $m$-tuple of nonnegative integers, set $\n{\alpha}:=\sum_{\jmath=1}^m\alpha_\jmath$ and $D^\alpha\psi:=D_1^{\alpha_1}D_2^{\alpha_2}\cdots D_m^{\alpha_m}\psi$ where $D_\jmath\psi=\frac{\partial\psi}{\partial z_\jmath}$ for $\jmath=1,2,\ldots,m$.
Define $\nu_\psi^0:U\to\mathbb{Z}_{\geq0}$ by $\nu_\psi^0(z):=\max\left\{\kappa:D^\alpha\psi(z)=0\right\}$ for all possible $\alpha\in\mathbb{Z}^m_{\geq0}$ with $\n{\alpha}<\kappa$, and write $\operatorname{supp}\,(\nu^0_\psi):=\overline{\left\{z\in U:\nu^0_\psi(z)>0\right\}}$.

For a meromorphic function $\varphi$ on $U$, there exist two coprime holomorphic functions $\psi_1,\psi_2$ on $U$ with $\varphi=\frac{\psi_1}{\psi_2}$ such that $\nu_\varphi^\infty:=\nu_{\psi_2}^0$ and $\operatorname{supp}\,(\nu^\infty_\varphi):=\operatorname{supp}\,(\nu^0_{\psi_2})$.

Take $\mu_0>0$ an integer or $\infty$.
For a meromorphic map $f:B(R_0)\to\mathbb{P}^n(\mathbb{C})$ with a reduced representation $\mathfrak{f}$, and a hypersurface $D$ in $\mathbb{P}^n(\mathbb{C})$ of degree $d$ with $Q$ its defining homogeneous polynomial, let $\nu^f_D:=\nu_{Q(\mathfrak{f})}^0$ be the intersection divisor associated with $f$ and $D$ on $B(R_0)\setminus\mathtt{I}_f$.
The \textsl{valence function} of $f$ regarding $D$, with truncation level $\mu_0$, is defined to be
\begin{equation}\label{e2.3}
N_f^{\mu_0}(r,r_0;D)=\int_{r_0}^r\frac{n_f^{\mu_0}(t;D)}{t}\,dt,
\end{equation}
where
\begin{equation}
n_f^{\mu_0}(t;D):=\left\{\begin{array}{ll}
\frac{1}{t^{2m-2}}\bigintsss_{\operatorname{supp}\,(\nu^f_D)\cap B(t)}\min\left\{\nu^f_D,\mu_0\right\}\upsilon_{m-1}&\mathrm{when}\hspace{2mm}m\geq2,\\\\
\sum_{\nm{z}<t}\min\left\{\nu^f_D(z),\mu_0\right\}&\mathrm{when}\hspace{2mm}m=1.\nonumber
\end{array}\right.
\end{equation}

The first main theorem says $N_f^{\mu_0}(r,r_0;D)\leq d\hspace{0.2mm}T_f(r,r_0)+O\left(1\right)$ (see \cite{Ha1,Ha2}).
Let $\hat{\delta}_{\mu_0}^f(D)$ be Nevanlinna's defect or its high dimensional extension by Stoll, that is defined by
\begin{equation}
\hat{\delta}_{\mu_0}^f(D)=1-\limsup_{r\to R_0}\frac{N_f^{\mu_0}(r,r_0;D)}{d\hspace{0.2mm}T_f(r,r_0)}.\nonumber
\end{equation}
When $\lim\limits_{r\to R_0}T_f(r,r_0)=\infty$, then \cite[Proposition 5.6]{Fu3} or \cite[Proposition 2.1]{RS1} yields
\begin{equation}\label{e2.4}
0\leq\delta_{\mu_0}^{f}(D)\leq\hat{\delta}_{\mu_0}^f(D)\leq1.
\end{equation}

Below, we recall two results of An, Quang and Thai \cite{AQT,QA}.
The first one is an extension to hypersurfaces of the celebrated Nochka weights \cite{No1,No2} concerning hyperplanes.

\begin{prop}\label{p2.1}{\rm(\cite[Lemma 3.3]{AQT} or \cite[Lemma 9]{QA})}
Assume that $V\subseteq\mathbb{P}^N(\mathbb{C})$ is an irreducible projective algebraic variety of dimension $n\left(n\leq N\right)$.
Let $D_1,D_2,\ldots,D_q$ be $q>2k-n+1\left(k\geq n\right)$ hypersurfaces in $\mathbb{P}^N(\mathbb{C})$ of common degree $d$ that are located in $k$-subgeneral position with respect to $V$.
Then, there exist $q$ rational numbers $0<\omega_1,\omega_2,\ldots,\omega_q\leq1$ such that
\vskip2pt\noindent {\bf(a.)} for $\varpi:=\max\limits_{j\in\left\{1,2,\ldots,q\right\}}\left\{\omega_j\right\}$, one has
\begin{equation}\label{e2.5}
\omega_j\leq\varpi=\frac{\sum_{j=1}^q\omega_j-n-1}{q-2k+n-1}\hspace{4mm}and\hspace{4mm}\frac{n+1}{2k-n+1}\leq\varpi\leq\frac{n}{k}
\footnote{Note this upper bound in the second estimate of \eqref{e2.5} was discovered by Toda \cite{Nj}.};
\end{equation}
\vskip2pt\noindent {\bf(b.)} for each subset $\mathcal{R}$ of $\left\{1,2,\ldots,q\right\}$ with $\#\mathcal{R}=k+1$, one has $\sum_{j\in\mathcal{R}}\omega_j\leq n+1$;
\vskip2pt\noindent {\bf(c.)} for $q$ arbitrarily given constants $E_1,E_2,\ldots,E_q\geq1$ and each set $\mathcal{R}$ as in {\bf(b.)}, there exists a subset $\mathcal{T}$ of $\mathcal{R}$ with $\#\mathcal{T}=\operatorname{rank}\left\{Q_j\right\}_{j\in\mathcal{T}}=n+1$ satisfying
\begin{equation}\label{e2.6}
\prod_{j\in\mathcal{R}}E_j^{\,\omega_j}\leq\prod_{j\in\mathcal{T}}E_j,
\end{equation}
where $Q_j$ is the defining homogeneous polynomial of $D_j$ in $\mathbb{P}^N(\mathbb{C})$ for $j=1,2,\ldots,q$.
\end{prop}

\begin{lem}\label{l2.2}{\rm(\cite[Lemma 4.2]{AQT} or \cite[Lemma 11]{QA})}
Under the same assumptions of Proposition \ref{p2.1}, for each subset $\mathcal{T}\subseteq\left\{1,2,\ldots,q\right\}$ with $\#\mathcal{T}=\operatorname{rank}\left\{Q_j\right\}_{j\in\mathcal{T}}=n+1$, there are $H_V(d)-n-1$ hypersurfaces $D^*_1,D^*_2,\ldots,D^*_{H_V(d)-n-1}$ in $\mathbb{P}^N(\mathbb{C})$ such that
\begin{equation*}
\operatorname{rank}\left\{\left\{Q_j\right\}_{j\in\mathcal{T}}\cup\left\{Q^*_i\right\}_{i=1}^{H_V(d)-n-1}\right\}=H_V(d).
\end{equation*}
Here, $Q_j$ and $Q^*_i$ are the homogeneous polynomials defining $D_j$ and $D^*_i$, respectively.
\end{lem}

\section{Proof of Theorem \ref{t1.2}}\label{Prof} 
\noindent First, it's interesting to notice the following consequence of our Theorem \ref{t1.2}.

\begin{thm}\label{t3.1}
Suppose $M$ is an $m$-dimensional complete K\"{a}hler manifold such that the universal covering of $M$ is biholomorphically isomorphic to a ball in $\mathbb{C}^m$, and $H_1,H_2,\ldots,H_q$ are $q$ hyperplanes in $\mathbb{P}^N(\mathbb{C})$ located in general position.
Let $f:M\to\mathbb{P}^N(\mathbb{C})$ be a meromorphic map, whose image spans a linear subspace with dimension $n$ not contained in any of $H_1,H_2,\ldots,H_q$, such that the {\bf condition $\mathbf{C}(\rho)$} is satisfied.
Then, one has the following defect relation
\begin{equation}\label{e3.1}
\sum_{j=1}^q\delta_n^f(H_j)\leq2N-n+1+\rho\,n\left(2N-n+1\right).
\end{equation}
\end{thm}

This is a Cartan-Nochka type result.
For the classical defect relation, the associated second main theorem was originally suggested by Cartan and proved by Nochka; for that with truncation, the associated second main theorem was initially proved by Fujimoto \cite[Theorem 3.2.12]{Fu5} and refined by Noguchi \cite[Theorem 3.1]{Nj} with a better estimate about error terms.

For each $j=1,2,\ldots,q$, set $Q_j$ to be the homogeneous polynomial of degree $d_j$ defining $D_j$ in $\mathbb{P}^N(\mathbb{C})$; replacing $Q_j$ by $Q_j^{d/d_j}$ when necessary, we may assume $Q_1,Q_2,\ldots,Q_q\in\mathcal{H}_d$, where from now on we use $d$ to represent the least common multiple of $d_1,d_2,\ldots,d_q$.

Now, we will proceed to prove Theorem \ref{t1.2} by considering two situations
\begin{equation}\label{e3.2}
\limsup_{r\to R_0}\frac{T_f(r,r_0)}{\log\frac{1}{R_0-r}}<\infty
\hspace{4mm}\mathrm{and}\hspace{4mm}\limsup_{r\to R_0}\frac{T_f(r,r_0)}{\log\frac{1}{R_0-r}}=\infty
\end{equation}
when the universal covering of $M$ is biholomorphic to a finite ball $B(R_0)$ in $\mathbb{C}^m$.

Let $\pi:\widetilde{M}\to M$ be the universal covering of $M$.
Then, $f\circ\pi:\widetilde{M}\to V$ is again algebraically non-degenerate since $f:M\to V$ is algebraically non-degenerate; also, one has $\delta_{H_V(d)-1}^{f}(D_j)\leq\delta_{H_V(d)-1}^{f\circ\pi}(D_j)$.
Hence, by lifting $f$ to the covering, we fix $M=B(1)$ subsequently.

Consider first the former case in \eqref{e3.2} that is more important.

Assume $\mathfrak{f}=\left(\mathtt{f}_0,\mathtt{f}_1,\ldots,\mathtt{f}_N\right)$ and $Q_j=\sum_{\beta\in\mathscr{I}_d}a_{j\beta}w^\beta$, where $\mathscr{I}_d$ is the set of $\left(N+1\right)$-tuples $\beta=\left(\beta_0,\beta_1,\ldots,\beta_N\right)\in\mathbb{Z}_{\geq0}^{N+1}$ with $\n{\beta}:=\sum_{\imath=0}^N\beta_\imath=d$ and $w^\beta:=w_0^{\beta_0}w_1^{\beta_1}\cdots w_N^{\beta_N}$.
For every $j=1,2,\ldots,q$, notice $\n{Q_j(\mathfrak{f})}=\n{\sum_{\beta\in\mathscr{I}_d}a_{j\beta}\mathfrak{f}^{\hspace{0.2mm}\beta}}
\leq\left(\sum_{\beta\in\mathscr{I}_d}\varrho_\beta\n{a_{j\beta}}\right)\nm{\mathfrak{f}}^d$ so that
\begin{equation}\label{e3.3}
\n{Q_j(\mathfrak{f})}\leq\varrho\nm{\mathfrak{f}}^d
\hspace{6mm}\mathrm{with}\hspace{2mm}\varrho:=\sum_{j=1}^q\sum_{\beta\in\mathscr{I}_d}\varrho_\beta\n{a_{j\beta}}>0.
\end{equation}

Fix a basis $\left\{\phi_1,\phi_2,\ldots,\phi_{H_V(d)}\right\}\subseteq\mathcal{H}_d$ of $\frac{\mathcal{H}_d}{\mathcal{H}_d\cap\mathcal{I}_V}$.
Because $f$ is algebraically non-degenerate, $F:=\left[\phi_1(\mathfrak{f}):\phi_2(\mathfrak{f}):\cdots:\phi_{H_V(d)}(\mathfrak{f})\right]:M\to\mathbb{P}^{H_V(d)-1}(\mathbb{C})$ is linearly non-degenerate.
In view of \cite[Proposition 4.5]{Fu3}, there exist $H_V(d)$ $m$-tuples $\alpha^l=\left(\alpha_1^l,\alpha_2^l,\ldots,\alpha_m^l\right)\in\mathbb{Z}^m_{\geq0}$ with
\begin{equation}\label{e3.4}
\n{\alpha^l}=\sum_{\jmath=1}^m\alpha^l_\jmath<l\hspace{4mm}\mathrm{and}\hspace{4mm}\sum_{l=1}^{H_V(d)}\n{\alpha^l}\leq\frac{H_V(d)\left(H_V(d)-1\right)}{2},
\end{equation}
such that the Wronskian $W_{\alpha^1\cdots\alpha^{H_V(d)}}(F)$ of $F$ is not identically zero on $M$, where
\begin{equation}\label{e3.5}
W_{\alpha^1\cdots\alpha^{H_V(d)}}(F):=\det\left(D^{\alpha^l}\phi_{\ell}(\mathfrak{f})\right)_{1\leq l,\hspace{0.2mm}\ell\leq H_V(d)}.
\end{equation}

For any subset $\mathcal{T}\subseteq\left\{1,2,\ldots,q\right\}$ with $\#\mathcal{T}=\operatorname{rank}\left\{Q_j\right\}_{j\in\mathcal{T}}=n+1$, use the hypersurfaces in Lemma \ref{l2.2} to define
$F_{\mathcal{T}}:=\left[\left\{Q_j(\mathfrak{f})\right\}_{j\in\mathcal{T}}:Q^*_1(\mathfrak{f}):\cdots:Q^*_{H_V(d)-n-1}(\mathfrak{f})\right]$ (by abuse of notations).
Then, there is a constant $C_{\mathcal{T}}\neq0$ such that $W_{\alpha^1\cdots\alpha^{H_V(d)}}(F_{\mathcal{T}})=C_{\mathcal{T}}\hspace{0.2mm}W_{\alpha^1\cdots\alpha^{H_V(d)}}(F)$.

Fix $\mathfrak{w}\in V\cap f(M)$.
Abusing the notation, $\mathfrak{w}=c\hspace{0.2mm}w$ for some $w\in\mathbb{C}^{N+1}\setminus\left\{0\right\}$ and all complex numbers $c\neq0$.
Select a subset $\mathcal{R}$ of $\left\{1,2,\ldots,q\right\}$ with $\#\mathcal{R}=k+1$ such that $\n{Q_j(w)}\leq\n{Q_s(w)}$ when $j\in\mathcal{R}$ and $s\in\left\{1,2,\ldots,q\right\}\setminus\mathcal{R}$; seeing the $k$-subgeneral position hypothesis \eqref{e1.5} and the continuity of $\frac{\n{Q_s(w)}^2}{\left(\n{w_0}^2+\n{w_1}^2+\cdots+\n{w_N}^2\right)^d}$, there exists a constant $\gamma_Y>0$ such that
\begin{equation}\label{e3.6}
\gamma_Y\nm{\mathfrak{f}(z)}^d\leq\min_{s\in\left\{1,2,\ldots,q\right\}\setminus\mathcal{R}}\n{Q_s(\mathfrak{f})(z)}
\end{equation}
for all $z\in f^{-1}(Y)\setminus\mathtt{I}_f$, where $Y$ is an appropriate open neighborhood of $\mathfrak{w}$ in $V$.

Take such a $z$ and set $E_j:=\frac{\varrho\nm{\mathfrak{f}(z)}^d}{\n{Q_j(\mathfrak{f})(z)}}\geq1$ for $j\in\mathcal{R}$; then, Proposition \ref{p2.1} - Parts {\bf(b.)}\&{\bf(c.)} yields a subset $\mathcal{T}$ of $\mathcal{R}$ with $\#\mathcal{T}=n+1$ such that the estimate \eqref{e2.6} holds.
Noting \eqref{e3.3}, \eqref{e3.6} and the estimate concerning these $E_j$ for $\mathcal{R},\mathcal{T}$, one observes that
\begin{equation*}
\begin{split}
&\,\frac{\nm{\mathfrak{f}(z)}^{d\sum_{j=1}^q\omega_j}\n{W_{\alpha^1\cdots\alpha^{H_V(d)}}(F)(z)}}
{\n{Q_1(\mathfrak{f})(z)}^{\omega_1}\cdots\n{Q_q(\mathfrak{f})(z)}^{\omega_q}}
\leq\prod_{j\in\mathcal{R}}\left(\frac{\varrho\nm{\mathfrak{f}(z)}^d}{\n{Q_j(\mathfrak{f})(z)}}\right)^{\omega_j}
\frac{\n{W_{\alpha^1\cdots\alpha^{H_V(d)}}(F)(z)}}
{\varrho^{\sum_{j\in\mathcal{R}}\omega_j}\gamma_Y^{\sum_{s\in\left\{1,2,\ldots,q\right\}\setminus\mathcal{R}}\omega_s}}\\
\leq&\,K\frac{\nm{\mathfrak{f}(z)}^{d\left(n+1\right)}\n{W_{\alpha^1\cdots\alpha^{H_V(d)}}(F)(z)}}{\prod_{j\in\mathcal{T}}\n{Q_j(\mathfrak{f})(z)}}
\leq K\frac{\nm{\mathfrak{f}(z)}^{dH_V(d)}\n{W_{\alpha^1\cdots\alpha^{H_V(d)}}(F_{\mathcal{T}})(z)}}
{\prod_{j\in\mathcal{T}}\n{Q_j(\mathfrak{f})(z)}\prod_{i=1}^{H_V(d)-n-1}\n{Q^*_i(\mathfrak{f})(z)}}.
\end{split}
\end{equation*}
Here, and hereafter, $K>0$ represents an absolute constant whose value may change from line to line but (in general) can be interpreted appropriately within the context.

For simplicity, put $\varphi:=\frac{W_{\alpha^1\cdots\alpha^{H_V(d)}}(F)}{Q_1^{\omega_1}(\mathfrak{f})\cdots Q_q^{\omega_q}(\mathfrak{f})}$ and $\aleph(F_{\mathcal{T}}):=\frac{W_{\alpha^1\cdots\alpha^{H_V(d)}}(F_{\mathcal{T}})}
{\prod_{j\in\mathcal{T}}Q_j(\mathfrak{f})\prod_{i=1}^{H_V(d)-n-1}Q^*_i(\mathfrak{f})}$.
Considering the compactness of $V$, we have
\begin{equation}\label{e3.7}
\nm{\mathfrak{f}(z)}^{d\left\{\sum_{j=1}^q\omega_j-H_V(d)\right\}}\n{\varphi(z)}\leq K\sum_{\mathcal{R},\mathcal{T}}\n{\aleph(F_{\mathcal{T}})(z)}
\hspace{6mm}\forall\hspace{2mm}z\in M\setminus\mathtt{I}_f.
\end{equation}
Here, the summation is taken over all the subsets $\mathcal{T}\subseteq\mathcal{R}\subseteq\left\{1,2,\ldots,q\right\}$ with $\#\mathcal{R}=k+1$ and $\#\mathcal{T}=n+1$.
Since $q,k,n$ are all finite, there can only be finitely many possibilities.

On the other hand, one may observe that
\begin{equation}\label{e3.8}
\nu_\varphi^\infty\leq\sum_{j=1}^q\omega_j\min\left\{\nu^f_{D_j},H_V(d)-1\right\}
\end{equation}
outside an analytic subset of codimension at least $2$.
As a matter of fact, when $\zeta\in M\setminus\mathtt{I}_f$ is a zero of some $Q_j(\mathfrak{f})$, it can be a zero of no more than $k+1$ functions $Q_j(\mathfrak{f})$ by \eqref{e1.5}.
Assume $Q_j(\mathfrak{f})$ vanishes at $\zeta$ for $j\in\tilde{\mathcal{R}}\subseteq\left\{1,2,\ldots,q\right\}$ with $\#\tilde{\mathcal{R}}=k+1$ yet $Q_s(\mathfrak{f})(\zeta)\neq0$ for $s\in\left\{1,2,\ldots,q\right\}\setminus\tilde{\mathcal{R}}$.
By virtue of Proposition \ref{p2.1} - Part {\bf(c.)}, putting $\tilde{E}_j:=\exp\left(\max\left\{\nu^f_{D_j}(\zeta)-H_V(d)+1,0\right\}\right)\geq1$ for $j\in\tilde{\mathcal{R}}$, there exists a subset $\tilde{T}$ of $\tilde{R}$ with $\#\tilde{T}=\operatorname{rank}\left\{Q_j\right\}_{j\in\tilde{T}}=n+1$ such that
\begin{equation*}
\sum_{j\in\tilde{\mathcal{R}}}\omega_j\max\left\{\nu^f_{D_j}(\zeta)-H_V(d)+1,0\right\}
\leq\sum_{j\in\tilde{\mathcal{T}}}\max\left\{\nu^f_{D_j}(\zeta)-H_V(d)+1,0\right\},
\end{equation*}
from which it follows that, in view of $\nu_{W_{\alpha^1\cdots\alpha^{H_V(d)}}(F)}^0=\nu_{W_{\alpha^1\cdots\alpha^{H_V(d)}}(F_{\tilde{\mathcal{T}}})}^0$,
\begin{equation*}
\sum_{j\in\tilde{\mathcal{R}}}\omega_j\max\left\{\nu^f_{D_j}(\zeta)-H_V(d)+1,0\right\}\leq\nu_{W_{\alpha^1\cdots\alpha^{H_V(d)}}(F)}^0(\zeta).
\end{equation*}
This estimate clearly leads to \eqref{e3.8} upon verifying the following computations at $\zeta$
\begin{equation*}
\begin{split}
&\,\nu_\varphi^\infty\leq\sum_{j=1}^q\omega_j\hspace{0.2mm}\nu^f_{D_j}-\nu_{W_{\alpha^1\cdots\alpha^{H_V(d)}}(F)}^0
=\sum_{j\in\tilde{\mathcal{R}}}\omega_j\hspace{0.2mm}\nu^f_{D_j}-\nu_{W_{\alpha^1\cdots\alpha^{H_V(d)}}(F)}^0\\
=&\sum_{j\in\tilde{\mathcal{R}}}\omega_j\left(\min\left\{\nu^f_{D_j},H_V(d)-1\right\}+\max\left\{\nu^f_{D_j}-H_V(d)+1,0\right\}\right)
-\nu_{W_{\alpha^1\cdots\alpha^{H_V(d)}}(F)}^0\\
\leq&\sum_{j\in\tilde{\mathcal{R}}}\omega_j\min\left\{\nu^f_{D_j},H_V(d)-1\right\}\leq\sum_{j=1}^q\omega_j\min\left\{\nu^f_{D_j},H_V(d)-1\right\}.
\end{split}
\end{equation*}

Next, we suppose that
\begin{equation}\label{e3.9}
\sum_{j=1}^q\omega_j\,\delta_{H_V(d)-1}^f(D_j)\leq H_V(d)+\frac{\rho}{d}\,H_V(d)\left(H_V(d)-1\right).
\end{equation}
When \eqref{e3.9} is true, then by \eqref{e1.2} and the first relation in \eqref{e2.5}, it yields that
\begin{equation*}
\sum_{j=1}^q\eta_j\,\omega_j\geq\varpi\left(q-2k+n-1\right)+n+1-H_V(d)-\frac{\rho}{d}\,H_V(d)\left(H_V(d)-1\right)
\end{equation*}
for all nonnegative constants $\eta_j\in\mathcal{A}\left(D_j,H_V(d)-1\right)$; that is,
\begin{equation*}
\sum_{j=1}^q\eta_j\geq q-2k+n-1+\frac{1}{\varpi}\left\{n+1-H_V(d)-\frac{\rho}{d}\,H_V(d)\left(H_V(d)-1\right)\right\}.
\end{equation*}
This further implies that
\begin{equation*}
\sum_{j=1}^q\left(1-\eta_j\right)\leq2k-n+1+\frac{1}{\varpi}\left\{H_V(d)-n-1+\frac{\rho}{d}\,H_V(d)\left(H_V(d)-1\right)\right\},
\end{equation*}
which, along with the lower bound in the second estimate of \eqref{e2.5}, leads to \eqref{e1.6}.

In the sequel, we show by contradiction the validity of \eqref{e3.9}.

Suppose it doesn't hold.
Then, by definition of non-integrated defect, there are nonnegative constants $\tilde{\eta}_j\in\mathcal{A}\left(D_j,H_V(d)-1\right)$ and continuous, pluri-subharmonic functions $\tilde{\mathtt{u}}_j\not\equiv-\infty$, for every $j=1,2,\ldots,q$, such that $e^{\tilde{\mathtt{u}}_j}\n{\psi_j}\leq\nm{\mathfrak{f}}^{d\tilde{\eta}_j}$ and
\begin{equation}\label{e3.10}
\sum_{j=1}^q\left(1-\tilde{\eta}_j\right)\omega_j>H_V(d)+\frac{\rho}{d}\,H_V(d)\left(H_V(d)-1\right).
\end{equation}
Here, $\psi_j$ is a nonzero holomorphic function that satisfies $\nu^0_{\psi_j}=\min\left\{\nu^f_{D_j},H_V(d)-1\right\}$.
Define $\mathtt{u}_j:=\tilde{\mathtt{u}}_j+\log\n{\psi_j}\not\equiv-\infty$ that is continuous and pluri-subharmonic, and satisfies $e^{\mathtt{u}_j}\leq\nm{\mathfrak{f}}^{d\tilde{\eta}_j}$.
So, for $\vartheta_1(z):=\log\n{z^{\boldsymbol{\alpha}}\varphi(z)}+\sum_{j=1}^q\omega_j\mathtt{u}_j(z)$ with $\boldsymbol{\alpha}:=\sum_{l=1}^{H_V(d)}\alpha^l\in\mathbb{Z}^m_{\geq0}$, seeing the preceding analyses and \eqref{e3.8}, one clearly deduces that $\vartheta_1$ is pluri-subharmonic on $M$.

Note we assume the {\bf condition $\mathbf{C}(\rho)$} is satisfied; that is, \eqref{e1.3} holds.
By \cite[p252, Remark]{Fu3}, there exists a continuous, pluri-subharmonic function $\vartheta_2\not\equiv-\infty$ such that $e^{\vartheta_2}\hspace{0.2mm}dV\leq\nm{\mathfrak{f}}^{2\rho}\upsilon_m$.
Here, and henceforth, we use $dV$ to denote the canonical volume form on $M$.

Set $t_0:=\frac{2\rho}{d\left\{\sum_{j=1}^q\left(1-\tilde{\eta}_j\right)\omega_j-H_V(d)\right\}}>0$ and write $\theta:=\vartheta_2+t_0\vartheta_1$.
Then, $\theta$ is pluri-subharmonic and thus a subharmonic function on $M=B(1)$.
In addition, one has
\begin{equation*}
\begin{split}
&\,e^\theta\hspace{0.2mm}dV=e^{\vartheta_2+t_0\vartheta_1}\hspace{0.2mm}dV\leq e^{t_0\vartheta_1}\nm{\mathfrak{f}}^{2\rho}\upsilon_m
=\n{z^{\boldsymbol{\alpha}}\varphi}^{t_0}e^{t_0\sum_{j=1}^q\omega_j\mathtt{u}_j}\nm{\mathfrak{f}}^{2\rho}\upsilon_m\\
\leq&\n{z^{\boldsymbol{\alpha}}\varphi}^{t_0}\nm{\mathfrak{f}}^{t_0d\sum_{j=1}^q\omega_j\tilde{\eta}_j+2\rho}\upsilon_m
=\n{z^{\boldsymbol{\alpha}}\varphi}^{t_0}\nm{\mathfrak{f}}^{t_0d\left\{\sum_{j=1}^q\omega_j-H_V(d)\right\}}\upsilon_m.
\end{split}
\end{equation*}
By \eqref{e3.4} and \eqref{e3.10}, we easily get $t_0\left(\sum_{l=1}^{H_V(d)}\n{\alpha^l}\right)<\varsigma<1$ for some constant $\varsigma>0$.
Therefore, recalling $\upsilon_m=2m\nm{z}^{2m-1}\sigma_m\wedge d\nm{z}$ and \eqref{e3.7}, we have
\begin{equation}\label{e3.11}
\begin{split}
&\,\int_Me^\theta\hspace{0.2mm}dV\leq\int_M\n{z^{\boldsymbol{\alpha}}}^{t_0}
\n{\varphi(z)\nm{\mathfrak{f}(z)}^{d\left\{\sum_{j=1}^q\omega_j-H_V(d)\right\}}}^{t_0}\upsilon_m\\
\leq&\,K\sum_{\mathcal{R},\mathcal{T}}\int_0^1r^{2m-1}\left(\int_{S(r)}\n{z^{\boldsymbol{\alpha}}\aleph(F_{\mathcal{T}})(z)}^{t_0}\sigma_m\right)dr\\
\leq&\,K\int_0^1r^{2m-1}\left(\frac{R^{2m-1}}{R-r}\hspace{0.2mm}T_F(R,r_0)\right)^\varsigma dr
\leq K\int_0^1\left(\frac{1}{R-r}\hspace{0.2mm}T_F(R,r_0)\right)^\varsigma dr
\end{split}
\end{equation}
for $r_0<r<R<1$, where we used \cite[Proposition 6.1]{Fu3} (see also \cite[Proposition 3.3]{RS1}) for the derivation of the third, or second last, estimate in \eqref{e3.11}.

Finally, seeing Hayman \cite[Lemma 2.4 (ii)]{Hwk} and letting $R=r+\frac{1-r}{e\hspace{0.2mm}T_F(r,r_0)}$, one has
\begin{equation*}
T_F(R,r_0)\leq2\hspace{0.2mm}T_F(r,r_0)\leq2d\hspace{0.2mm}T_f(r,r_0)
\end{equation*}
outside a set with finite logarithmic measure.
Recall we assumed the case $\limsup\limits_{r\to1}\frac{T_f(r,r_0)}{\log\frac{1}{1-r}}<\infty$ in \eqref{e3.2}.
The preceding analyses combined with \cite[Proposition 5.5]{Fu2} yields that
\begin{equation}\label{e3.12}
\begin{split}
&\,\int_Me^\theta\hspace{0.2mm}dV\leq K\int_0^1\left(\frac{2}{\frac{1-r}{e\hspace{0.2mm}T_F(r,r_0)}}\hspace{0.2mm}T_F(r,r_0)\right)^\varsigma dr
\leq K\int_0^1\left(\frac{2\hspace{0.2mm}d^2\hspace{0.2mm}e}{1-r}\hspace{0.2mm}T^2_f(r,r_0)\right)^\varsigma dr\\
\leq&\,K\int_0^1\frac{1}{\left(1-r\right)^\varsigma}\left(\log\frac{1}{1-r}\right)^{2\varsigma}dr
=\frac{K}{\left(1-\varsigma\right)^{2\varsigma+1}}\,\Gamma(2\varsigma+1)<\infty.
\end{split}
\end{equation}

This result however would contradict Yau \cite{Yst} and Karp \cite[Theorem B]{Ka}, as $M=B(1)$ has infinite volume with respect to the given complete K\"{a}hler metric; see \cite[p1147]{RS1}.

From now on, we shall consider the latter case in \eqref{e3.2} and the situation when the universal covering of $M$ is biholomorphic to $\mathbb{C}^m$ simultaneously, since both may be treated essentially in the same way through traditional defect relation and \eqref{e2.4}.
As can be seen from the following discussions, we don't need $f$ to satisfy the growth {\bf condition $\mathbf{C}(\rho)$} in these settings.

Noting the description below \eqref{e3.2}, we without loss of generality assume $M=B(R_0)$ for some $0<R_0\leq\infty$ afterwards.
Moreover, when $R_0=\infty$, we can use the flat metric to see $\operatorname{Ric}\boldsymbol{\omega}\equiv0$; that is, all meromorphic maps $f:\mathbb{C}^m\to V$ satisfy the {\bf condition $\mathbf{C}(0)$} automatically.

\begin{prop}\label{p3.2}
Under the same hypotheses of Theorem \ref{t1.2} concerning the algebraic variety $V$ in $\mathbb{P}^N(\mathbb{C})$ and the hypersurfaces $D_1,D_2,\ldots,D_q$ in $\mathbb{P}^N(\mathbb{C})$, let $f:B(R_0)\left(\subseteq\mathbb{C}^m\right)\to V$ be an algebraically non-degenerate meromorphic map with $0<R_0\leq\infty$.
Then, one has
\begin{equation}\label{e3.13}
\left\{q-\frac{2k-n+1}{n+1}H_V(d)\right\}T_f(r,r_0)\leq\sum_{j=1}^q\frac{1}{d_j}N_f^{H_V(d)-1}(r,r_0;D_j)+S_f(r,r_0),
\end{equation}
where $S_f(r,r_0)\geq0$ satisfies $S_f(r,r_0)\leq K\left\{\log^+T_f(r,r_0)+\log^+r\right\}$ for all $r\in\left(r_0,\infty\right)$ outside a set of finite linear measure when $R_0=\infty$ and
\begin{equation}\label{e3.14}
S_f(r,r_0)\leq K\log^+T_f(r,r_0)+\frac{2k-n+1}{2\hspace{0.2mm}d\left(n+1\right)}H_V(d)\left(H_V(d)-1\right)\log^+\frac{1}{R_0-r}
\end{equation}
for all $r\in\left(r_0,R_0\right)$ outside a set of finite logarithmic measure when $R_0<\infty$.
\end{prop}

\begin{proof}
Like the first case in \eqref{e3.2}, by \eqref{e3.4}, \eqref{e3.7} and the argument in \eqref{e3.11}, one has
\begin{equation*}
\int_{S(r)}\n{z^{\boldsymbol{\alpha}}\varphi(z)\nm{\mathfrak{f}(z)}^{d\left\{\sum_{j=1}^q\omega_j-H_V(d)\right\}}}^{\check{t}_0}\sigma_m
\leq K\left(\frac{R^{2m-1}}{R-r}\hspace{0.2mm}T_F(R,r_0)\right)^{\check{\varsigma}}
\end{equation*}
for $r_0<r<R<R_0$, which further implies that, applying the concavity of logarithm,
\begin{equation}\label{e3.15}
\begin{split}
&\,\int_{S(r)}\log\n{z^{\boldsymbol{\alpha}}}\sigma_m+\int_{S(r)}\log\n{\varphi}\sigma_m
+\int_{S(r)}\log\nm{\mathfrak{f}}^{d\left\{\sum_{j=1}^q\omega_j-H_V(d)\right\}}\sigma_m\\
\leq&\,\frac{\check{\varsigma}}{\check{t}_0}\log^+\frac{1}{R-r}+K\left\{\log^+T_F(R,r_0)+\log^+R\right\}.
\end{split}
\end{equation}
Here, $\check{t}_0,\check{\varsigma}>0$ are arbitrarily given constants satisfying $\check{t}_0\left(\sum_{l=1}^{H_V(d)}\n{\alpha^l}\right)<\check{\varsigma}<1$.
Besides, use Jensen's formula and \eqref{e3.8} to derive that
\begin{equation*}
\int_{S(r)}\log\n{z^{\boldsymbol{\alpha}}\times\varphi(z)}\sigma_m\geq-\sum_{j=1}^q\omega_jN_f^{H_V(d)-1}(r,r_0;D_j)+O\left(1\right),
\end{equation*}
which combined with the first relation in \eqref{e2.5} and \eqref{e3.15} altogether leads to
\begin{equation*}
\begin{split}
&\left\{\varpi\left(q-2k+n-1\right)+n+1-H_V(d)\right\}T_f(r,r_0)\leq\sum_{j=1}^q\frac{\omega_j}{d}N_f^{H_V(d)-1}(r,r_0;D_j)\\
&+\frac{H_V(d)\left(H_V(d)-1\right)}{2\hspace{0.2mm}d}\log^+\frac{1}{R-r}+K\left\{\log^+T_F(R,r_0)+\log^+R\right\}
\end{split}
\end{equation*}
when $\frac{\check{\varsigma}}{\check{t}_0}$ approaches $\frac{H_V(d)\left(H_V(d)-1\right)}{2}$ from the above.
Since $d_j\leq d$, $\omega_j\leq\varpi$ and $\frac{1}{\varpi}\leq\frac{2k-n+1}{n+1}$, \eqref{e3.13} follows immediately from the above inequality with
\begin{equation*}
S_f(r,r_0):=\frac{1}{2\hspace{0.2mm}d\hspace{0.2mm}\varpi}H_V(d)\left(H_V(d)-1\right)\log^+\frac{1}{R-r}+K\left\{\log^+T_F(R,r_0)+\log^+R\right\}.
\end{equation*}
The remaining estimates about $S_f(r,r_0)$ appear to be exactly the same as those, for instance, in \cite[Proposition 6.2]{Fu3} or \cite[Theorem 4.5]{RS1} by virtue of \cite[Lemma 2.4]{Hwk}.
\end{proof}

A natural consequence of Proposition \ref{p3.2} is the standard defect relation
\begin{equation}\label{e3.16}
\sum_{j=1}^q\hat{\delta}_{H_V(d)-1}^f(D_j)\leq\frac{2k-n+1}{n+1}H_V(d),
\end{equation}
provided either $R_0=\infty$ and $f$ is transcendental\footnote{Notice when $f$ is rational, then one can choose $S_f(r,r_0)=O\left(1\right)$ to have \eqref{e3.16}.}, or $R_0<\infty$ and $\limsup\limits_{r\to R_0}\frac{T_f(r,r_0)}{\log\frac{1}{R_0-r}}=\infty$.
Thus, \eqref{e1.6} follows from \eqref{e2.4} and \eqref{e3.16} so that our proof is finished completely.

\section{Some Related Uniqueness Results}\label{Uni} 
\noindent In 1986, Fujimoto \cite{Fu4} generalized the well-known five-value theorem of Nevanlinna to the situation of meromorphic maps over a complete, connected K\"{a}hler manifold $M$, whose universal covering is biholomorphic to a finite ball in $\mathbb{C}^m$, into $\mathbb{P}^n(\mathbb{C})$ that satisfy the growth {\bf condition $\mathbf{C}(\rho)$} and share hyperplanes; other closely related results can be found in \cite{RS2,Ya}.

In this last section, under the same setting as this result of Fujimoto, we use the techniques in the proof of Theorem \ref{t1.2} to describe two uniqueness results regarding hypersurfaces located in $k$-subgeneral position, following essentially the approach applied in \cite{Fu4,RS2,Ya}.

Considering the comments made in \cite[Section 5]{Fu4}, we will without loss of generality suppose that either $M=B(1)\subseteq\mathbb{C}^m$ (finite ball covering of $M$) or $M=\mathbb{C}^m$ subsequently.

In fact, when $f,g:M\to V$ are the given meromorphic maps, then $f\circ\pi,g\circ\pi:\widetilde{M}\to V$ will satisfy all the hypotheses as meromorphic maps over the lifted, complete universal covering $\widetilde{M}$ of $M$.
Since $f\circ\pi\equiv g\circ\pi$ on $\widetilde{M}$ implies $f\equiv g$ on $M$, we simply assume $M=\widetilde{M}$.

\begin{thm}\label{t4.1}
Assume $V\subseteq\mathbb{P}^N(\mathbb{C})$ is an irreducible projective algebraic variety of dimension $n\left(\leq N\right)$.
Let $f,g:B(1)\left(\subseteq\mathbb{C}^m\right)\to V$ be two algebraically non-degenerate meromorphic maps, both satisfying the {\bf condition $\mathbf{C}(\rho)$}.
Let $D_1,D_2,\ldots,D_q$ be $q$ hypersurfaces in $\mathbb{P}^N(\mathbb{C})$ of degrees $d_1,d_2,\ldots,d_q$, located in $k$-subgeneral position $\left(k\geq n\right)$ with respect to $V$.
Suppose further that $\limsup\limits_{r\to1}\frac{T_f(r,r_0)+T_g(r,r_0)}{\log\frac{1}{1-r}}<\infty$ and $f,g$ satisfy the following conditions
\vskip2pt\noindent {\bf(1.)} $f^{-1}(D_j)=g^{-1}(D_j)$ for $j=1,2,\ldots,q$,
\vskip2pt\noindent {\bf(2.)} $f=g$ on $\bigcup\limits_{j=1}^qf^{-1}(D_j)$,
\vskip2pt\noindent {\bf(3.)} $f^{-1}(D_j\cap D_{j^{\prime}})$ has dimension at most $m-2$ for $1\leq j\neq j^{\prime}\leq q$.
\vskip2pt\noindent Then, one has $f\equiv g$ provided, for the least common multiple $d$ of $d_1,d_2,\ldots,d_q$,
\begin{equation}\label{e4.1}
q>\frac{2k-n+1}{n+1}\left\{H_V(d)+\frac{\rho}{d}\,H_V(d)\left(H_V(d)-1\right)\right\}+\frac{2}{d}\left(H_V(d)-1\right).
\end{equation}
\end{thm}

\begin{proof}
Assume $f=\left[\mathtt{f}_0:\mathtt{f}_1:\cdots:\mathtt{f}_N\right]$ and $g=\left[\mathtt{g}_0:\mathtt{g}_1:\cdots:\mathtt{g}_N\right]$, with reduced representations $\mathfrak{f}=\left(\mathtt{f}_0,\mathtt{f}_1,\ldots,\mathtt{f}_N\right)$ and $\mathfrak{g}=\left(\mathtt{g}_0,\mathtt{g}_1,\ldots,\mathtt{g}_N\right)$.
Suppose in the following $f\not\equiv g$.
Then, there exist at least two distinct indices $0\leq\imath\neq\breve{\imath}\leq N$ such that the holomorphic function
$\chi:=\mathtt{f}_\imath\mathtt{g}_{\breve{\imath}}-\mathtt{f}_{\breve{\imath}}\mathtt{g}_\imath$ is not identically zero and satisfies $\n{\chi}\leq2\nm{\mathfrak{f}}\nm{\mathfrak{g}}$ on $M=B(1)$.

Employ the previous notations to have $F$ as before and $G:=\left[\phi_1(\mathfrak{g}):\phi_2(\mathfrak{g}):\cdots:\phi_{H_V(d)}(\mathfrak{g})\right]$, both being linearly non-degenerate maps to $\mathbb{P}^{H_V(d)-1}(\mathbb{C})$.
Thus, one finds two sets of $H_V(d)$ $m$-tuples $\alpha^l,\tilde{\alpha}^l\in\mathbb{Z}^m_{\geq0}$ with \eqref{e3.4} satisfied for each one, and
$W_{\alpha^1\cdots\alpha^{H_V(d)}}(F)\times W_{\tilde{\alpha}^1\cdots\tilde{\alpha}^{H_V(d)}}(G)\not\equiv0$ with
$W_{\tilde{\alpha}^1\cdots\tilde{\alpha}^{H_V(d)}}(G):=\det\left(D^{\tilde{\alpha}^l}\phi_{\ell}(\mathfrak{g})\right)_{1\leq l,\hspace{0.2mm}\ell\leq H_V(d)}$. Besides, for every subset $\mathcal{T}\subseteq\left\{1,2,\ldots,q\right\}$ with $\#\mathcal{T}=\operatorname{rank}\left\{Q_j\right\}_{j\in\mathcal{T}}=n+1$, use the hypersurfaces in Lemma \ref{l2.2} to define $G_{\mathcal{T}}$ similarly, and there is a constant $\tilde{C}_{\mathcal{T}}\neq0$ such that $W_{\tilde{\alpha}^1\cdots\tilde{\alpha}^{H_V(d)}}(G_{\mathcal{T}})=\tilde{C}_{\mathcal{T}}\hspace{0.2mm}W_{\tilde{\alpha}^1\cdots\tilde{\alpha}^{H_V(d)}}(G)$.
Recall $\varphi=\frac{W_{\alpha^1\cdots\alpha^{H_V(d)}}(F)}{Q_1^{\omega_1}(\mathfrak{f})\cdots Q_q^{\omega_q}(\mathfrak{f})}$ and $\aleph(F_{\mathcal{T}})=\frac{W_{\alpha^1\cdots\alpha^{H_V(d)}}(F_{\mathcal{T}})}
{\prod_{j\in\mathcal{T}}Q_j(\mathfrak{f})\prod_{i=1}^{H_V(d)-n-1}Q^*_i(\mathfrak{f})}$.
Analogously, set $\tilde{\varphi}:=\frac{W_{\tilde{\alpha}^1\cdots\tilde{\alpha}^{H_V(d)}}(G)}{Q_1^{\omega_1}(\mathfrak{g})\cdots Q_q^{\omega_q}(\mathfrak{g})}$ and $\tilde{\aleph}(G_{\mathcal{T}}):=\frac{W_{\tilde{\alpha}^1\cdots\tilde{\alpha}^{H_V(d)}}(G_{\mathcal{T}})} {\prod_{j\in\mathcal{T}}Q_j(\mathfrak{g})\prod_{i=1}^{H_V(d)-n-1}Q^*_i(\mathfrak{g})}$.
Then, one has \eqref{e3.7} and
\begin{equation}\label{e4.2}
\nm{\mathfrak{g}(z)}^{d\left\{\sum_{j=1}^q\omega_j-H_V(d)\right\}}\n{\tilde{\varphi}(z)}
\leq K\sum_{\mathcal{R},\mathcal{T}}\n{\tilde{\aleph}(G_{\mathcal{T}})(z)}\hspace{6mm}\forall\hspace{2mm}z\in M\setminus\mathtt{I}_g.
\end{equation}

Now, it is routine to see our condition \eqref{e4.1} and \eqref{e2.5} imply that
\begin{equation}\label{e4.3}
\sum\limits_{j=1}^q\omega_j>H_V(d)+\frac{2\hspace{0.2mm}\varpi}{d}\left(H_V(d)-1\right)+\frac{\rho}{d}\,H_V(d)\left(H_V(d)-1\right).
\end{equation}

From our hypotheses, we know $\chi(z)=0$ for all $z\in\bigcup\limits_{j=1}^qf^{-1}(D_j)$.
As $\omega_j\leq\varpi$, we can infer that $\nu_\varphi^\infty,\nu_{\tilde{\varphi}}^\infty\leq\varpi\left(H_V(d)-1\right)\nu_\chi^0$ and thus $\varphi\hspace{0.2mm}\chi^{\varpi\left(H_V(d)-1\right)},\tilde{\varphi}\hspace{0.2mm}\chi^{\varpi\left(H_V(d)-1\right)}$ are both holomorphic functions on $B(1)$.
Recall the K\"{a}hler form $\boldsymbol{\omega}=\frac{\sqrt{-1}}{2}\sum_{i,j}h_{i\bar{j}}\,dz_i\wedge d{\bar{z}_j}$ on $B(1)$.
By assumption, there exist two continuous, pluri-subharmonic functions $\tau_1,\tau_2\not\equiv-\infty$ such that
\begin{equation*}
e^{\tau_1}\sqrt{\det(h_{i\bar{j}})}\leq\nm{\mathfrak{f}}^\rho\hspace{4mm}\mathrm{and}\hspace{4mm}e^{\tau_2}\sqrt{\det(h_{i\bar{j}})}\leq\nm{\mathfrak{g}}^\rho.
\end{equation*}
Take $\tau:=\log\n{z^{\boldsymbol{\alpha}+\tilde{\boldsymbol{\alpha}}}
\varphi\hspace{0.2mm}\tilde{\varphi}\hspace{0.2mm}\chi^{2\varpi\left(H_V(d)-1\right)}}^{\hat{t}_0}$ for $\hat{t}_0:=\frac{\rho}{d\left\{\sum_{j=1}^q\omega_j-H_V(d)\right\}-2\varpi\left(H_V(d)-1\right)}>0$ with $\boldsymbol{\alpha}=\sum_{l=1}^{H_V(d)}\alpha^l,\tilde{\boldsymbol{\alpha}}:=\sum_{l=1}^{H_V(d)}\tilde{\alpha}^l\in\mathbb{Z}^m_{\geq0}$.
Then, $\tau$ is pluri-subharmonic and one has
\begin{equation*}
\begin{split}
&\,\det(h_{i\bar{j}})\,e^{\tau+\tau_1+\tau_2}\leq\n{z^{\boldsymbol{\alpha}}\varphi}^{\hat{t}_0}\n{z^{\tilde{\boldsymbol{\alpha}}}\tilde{\varphi}}^{\hat{t}_0}
\n{\chi}^{2\hat{t}_0\varpi\left(H_V(d)-1\right)}\nm{\mathfrak{f}}^\rho\nm{\mathfrak{g}}^\rho\\
\leq&\,K\n{z^{\boldsymbol{\alpha}}\varphi}^{\hat{t}_0}\nm{\mathfrak{f}}^{\rho+2\hat{t}_0\varpi\left(H_V(d)-1\right)}
\n{z^{\tilde{\boldsymbol{\alpha}}}\tilde{\varphi}}^{\hat{t}_0}\nm{\mathfrak{g}}^{\rho+2\hat{t}_0\varpi\left(H_V(d)-1\right)}\\
=&\,K\n{z^{\boldsymbol{\alpha}}\varphi}^{\hat{t}_0}\nm{\mathfrak{f}}^{d\hat{t}_0\left\{\sum_{j=1}^q\omega_j-H_V(d)\right\}}
\n{z^{\tilde{\boldsymbol{\alpha}}}\tilde{\varphi}}^{\hat{t}_0}\nm{\mathfrak{g}}^{d\hat{t}_0\left\{\sum_{j=1}^q\omega_j-H_V(d)\right\}}.
\end{split}
\end{equation*}

Via \eqref{e4.3}, we get $\hat{t}_0H_V(d)\left(H_V(d)-1\right)<\hat{\varsigma}<1$ for some constant $\hat{\varsigma}>0$.
So, seeing $dV=c_m\det(h_{i\overline{j}})\,\upsilon_m$ for an absolute constant $c_m>0$, \eqref{e3.4}, \eqref{e3.7} and \eqref{e4.2}, we have
\begin{equation}\label{e4.4}
\begin{split}
\int_Me^{\tau+\tau_1+\tau_2}\hspace{0.2mm}dV
\leq&\,K\left(\int_M\n{z^{\boldsymbol{\alpha}}}^{2\hat{t}_0}
\n{\varphi(z)\nm{\mathfrak{f}(z)}^{d\left\{\sum_{j=1}^q\omega_j-H_V(d)\right\}}}^{2\hat{t}_0}\upsilon_m\right)^{\frac{1}{2}}\\
&\times\left(\int_M\n{z^{\tilde{\boldsymbol{\alpha}}}}^{2\hat{t}_0}
\n{\tilde{\varphi}(z)\nm{\mathfrak{g}(z)}^{d\left\{\sum_{j=1}^q\omega_j-H_V(d)\right\}}}^{2\hat{t}_0}\upsilon_m\right)^{\frac{1}{2}}\\
\leq&\,K\left\{\sum_{\mathcal{R},\mathcal{T}}
\int_0^1r^{2m-1}\left(\int_{S(r)}\n{z^{\boldsymbol{\alpha}}\aleph(F_{\mathcal{T}})(z)}^{2\hat{t}_0}\sigma_m\right)dr\right\}^{\frac{1}{2}}\\
&\times\left\{\sum_{\mathcal{R},\mathcal{T}}
\int_0^1r^{2m-1}\left(\int_{S(r)}\n{z^{\tilde{\boldsymbol{\alpha}}}\tilde{\aleph}(G_{\mathcal{T}})(z)}^{2\hat{t}_0}\sigma_m\right)dr\right\}^{\frac{1}{2}}\\
\leq&\,K\int_0^1\frac{1}{\left(1-r\right)^{\hat{\varsigma}}}\left(\log\frac{1}{1-r}\right)^{2\hat{\varsigma}}dr
=\frac{K}{\left(1-\hat{\varsigma}\right)^{2\hat{\varsigma}+1}}\,\Gamma(2\hat{\varsigma}+1)<\infty
\end{split}
\end{equation}
by H\"{o}lder's inequality, where a parallel argument concerning \eqref{e3.11} and \eqref{e3.12} is used to derive \eqref{e4.4}.
This contradicts the results of Yau \cite{Yst} and Karp \cite{Ka}, and thus $f\equiv g$.
\end{proof}

Finally, we describe a uniqueness result when the growth {\bf condition $\mathbf{C}(\rho)$} is dropped.
Since it follows directly from the discussions in \cite[Section 4]{Fu4} (see also \cite[Section 3]{RS2} or \cite[Theorem 4.2]{Ya}) and our Proposition \ref{p3.2} (in particular \eqref{e3.14}), we only outline its proof.

\begin{prop}\label{p4.2}
Under the same hypotheses of Theorem \ref{t4.1} concerning the algebraic variety $V$ in $\mathbb{P}^N(\mathbb{C})$ and the hypersurfaces $D_1,D_2,\ldots,D_q$ in $\mathbb{P}^N(\mathbb{C})$, suppose $f,g:B(R_0)\left(\subseteq\mathbb{C}^m\right)\to V$ are algebraically non-degenerate meromorphic maps satisfying the conditions {\bf(1.)\hspace{0.2mm}-\hspace{0.2mm}(3.)}.
Fix $d$ the least common multiple of $d_1,d_2,\ldots,d_q$.
Then, one has $f\equiv g$ provided either
\begin{equation}\label{e4.5}
q>\frac{2k-n+1}{n+1}H_V(d)+\frac{2}{d}\left(H_V(d)-1\right)
\end{equation}
when $R_0=\infty$ or
\begin{equation}\label{e4.6}
q>\frac{2k-n+1}{n+1}\left\{H_V(d)+\frac{\lambda}{d}\,H_V(d)\left(H_V(d)-1\right)\right\}+\frac{2}{d}\left(H_V(d)-1\right)
\end{equation}
when $R_0=1$ with $\lambda:=\liminf\limits_{r\to1}\frac{\log\frac{1}{1-r}}{T_f(r,r_0)+T_g(r,r_0)}$ outside a set of finite logarithmic measure.
\end{prop}

\begin{proof}
From the derivation of \eqref{e3.13} and the facts that $\chi=0$ on $\bigcup\limits_{j=1}^qf^{-1}(D_j)$ and $T_\chi(r,r_0)\leq T_f(r,r_0)+T_g(r,r_0)$, one has for the valence function $N\left(r,r_0;\frac{1}{\chi}\right)$ of zeros of $\chi$
\begin{equation*}
\begin{split}
&\left\{q-\frac{2k-n+1}{n+1}H_V(d)\right\}\left\{T_f(r,r_0)+T_g(r,r_0)\right\}\\
\leq\,&\frac{2}{d}\left(H_V(d)-1\right)N\left(r,r_0;\frac{1}{\chi}\right)+S_f(r,r_0)+S_g(r,r_0)
\end{split}
\end{equation*}
when we suppose $f\not\equiv g$; that is, considering the first main theorem,
\begin{equation*}
q\leq\frac{2k-n+1}{n+1}H_V(d)+\frac{2}{d}\left(H_V(d)-1\right)+\liminf\limits_{r\to R_0}\frac{S_f(r,r_0)+S_g(r,r_0)}{T_f(r,r_0)+T_g(r,r_0)}.
\end{equation*}
If $R_0=\infty$, a contradiction against \eqref{e4.5} follows\footnote{Recall when $f,g$ are rational, then $S_f(r,r_0)=S_g(r,r_0)=O\left(1\right)$.}; on the other hand, if $R_0=1$, \eqref{e3.14} yields a contradiction against \eqref{e4.6} as $\liminf\limits_{r\to1}\frac{S_f(r,r_0)+S_g(r,r_0)}{T_f(r,r_0)+T_g(r,r_0)}\leq\frac{\lambda\left(2k-n+1\right)}{d\left(n+1\right)}H_V(d)\left(H_V(d)-1\right)$.
\end{proof}


\bibliographystyle{amsplain}

\end{document}